\documentclass[11pt, reqno]{amsart}
\usepackage{comment}
\usepackage{xcolor}
\usepackage{stmaryrd}
\usepackage{amsfonts}
\usepackage{amssymb}
\usepackage{amsmath,amsthm}
\usepackage{latexsym}
\usepackage{amscd}
\usepackage{esint}
\usepackage{mathrsfs}
\usepackage{dsfont}
\usepackage{setspace}
\usepackage{enumitem}
\usepackage[margin=1.25in]{geometry}
\usepackage{hyperref}
\usepackage{comment}

\newtheorem{prop}{Proposition}[section]
\newtheorem{thm}[prop]{Theorem}
\newtheorem{lemm}[prop]{Lemma}
\newtheorem{coro}[prop]{Corollary}

\newtheorem{claim}[prop]{Claim}
\newtheorem*{claim*}{Claim}

\theoremstyle{definition}
\newtheorem{defi}[prop]{Definition}
\newtheorem{rmk}[prop]{Remark}

\newcommand{\CC}{\mathbb{C}}

\newcommand{\PP}{\mathbb{P}}

\newcommand{\RR}{\mathbb{R}}

\newcommand{\ZZ}{\mathbb{Z}}

\newcommand{\cC}{\mathcal C}

\newcommand{\cK}{\mathcal K}

\DeclareMathOperator{\tr}{tr}

\DeclareMathOperator{\im}{Im}

\DeclareMathOperator{\supp}{supp}

\DeclareMathOperator{\re}{Re}

\DeclareMathOperator{\Ric}{Ric}

\newcommand{\ep}{\varepsilon}

\setlist[enumerate]{leftmargin = 2em}
\allowdisplaybreaks

\numberwithin{equation}{section}

\title[Stable solutions to the abelian Yang--Mills--Higgs equations]{Stable solutions to the abelian Yang--Mills--Higgs equations on $S^2$ and $T^2$}
\author{Da Rong Cheng}
\address{Department of Mathematics, University of Chicago, Chicago, IL 60637}
\email{chengdr@uchicago.edu}

\begin{document}

\maketitle 
\begin{abstract} We show under natural assumptions that stable solutions to the abelian Yang--Mills--Higgs equations on Hermitian line bundles over the round $2$-sphere actually satisfy the vortex equations, which are a first-order reduction of the (second-order) abelian Yang--Mills--Higgs equations. We also obtain a similar result for stable solutions on a flat $2$-torus. Our method of proof comes from the work of Bourguignon--Lawson~\cite{BL} concerning stable $SU(2)$ Yang--Mills connections on compact homogeneous $4$-manifolds. 
\end{abstract}


\section{Introduction}
Let $\Sigma$ be an oriented surface equipped with a Riemannian metric $g$, and suppose $L$ is a complex line bundle over $\Sigma$ equipped with a Hermitian metric $\langle \cdot, \cdot \rangle$, so that for all $x \in \Sigma$, we have
\[
\langle \alpha \xi, \beta \eta \rangle = \alpha\overline{\beta}\langle \xi, \eta \rangle \text{ for all $\alpha, \beta \in \CC$ and $\xi, \eta \in L_{x}$.}
\]
\vskip 1mm
Given $\ep > 0$, we are interested in the following self-dual abelian Yang--Mills--Higgs action, which takes a section $u: \Sigma \to L$ and a metric connection $\nabla_{A}$ on $L$ as variables:
\[
E_{\ep}(u, \nabla_{A}) = \int_{\Sigma} \ep^{2} |F_{A}|^{2} + |\nabla_{A}u|^{2} + \frac{(1 - |u|^{2})^{2}}{4\ep^{2}} d\mu_{g}.
\]
Here $F_{A}$ denotes the curvature of the connection $\nabla_{A}$, and $d\mu_{g}$ is the volume form on $\Sigma$. Note that since $\nabla_{A}$ is a $U(1)$-connection, $F_A$ is in fact a $2$-form with values in $\sqrt{-1}\RR$. Historically, functionals of this type originated with the Ginzburg--Landau theory of superconductivity, and later made their way into elementary particle physics, where $U(1)$ may be replaced by other groups depending on the situation. The reader interested in a succinct account of the physical backgrounds may consult, for example~\cite[Chapter 1]{JT}.

A straightforward computation yields the Euler--Lagrange equations of $E_{\ep}$:
\begin{equation}\label{eq:YMH}
\left\{
\begin{array}{cl}
\nabla^{\ast}_{A} \nabla_{A}u &= \frac{1 - |u|^{2}}{2\ep^{2}}u, \\
\ep^{2} d^{\ast} F_{A} &= -\sqrt{-1}\re\langle \sqrt{-1}u, \nabla_A u \rangle = \sqrt{-1}\im\langle u, \nabla_A u \rangle,  
\end{array}
\right.
\end{equation}
where $\nabla_{A}^{\ast}$ is the formal adjoint of $\nabla_A$, and both sides of the second equation are $1$-forms valued in $\sqrt{-1}\RR$. Note also that the second equation is not an elliptic equation for the connection $\nabla_A$. This can be attributed to the gauge invariance of $E_{\ep}$, where
\begin{equation}\label{eq:gauge-invariance}
E_{\ep}(u, \nabla_A) = E_{\ep}(s\cdot u, \nabla_A - ds \cdot s^{-1}) \text{ for any }s : \Sigma \to U(1) \simeq S^{1}.
\end{equation}

The system~\eqref{eq:YMH} and its relatives have been the subject of extensive study, and there is by now a large literature on them, which we do not attempt to survey here. The interested reader is referred to the monographs~\cite{BBH, PT, SaSe} and the references therein. Below, we will focus on one particular aspect of the second-order equations~\eqref{eq:YMH}, namely that they admit special solutions given by first-order equations which arise from rewriting the functional $E_{\ep}$ in a particular way. Specifically, it was observed by Bogomol'nyi~\cite{Bol} that, with the choice of potential $\frac{(1 - |u|^{2})^{2}}{4}$ in the definition of $E_{\ep}$, the functional can be split into two parts whose difference, after an integration by parts, is a topological invariant, as follows:
\begin{align}\label{eq:Bogol}
E_{\ep}(u, \nabla_A) =\ & \int_{\Sigma} \frac{1}{4}|\nabla_A u - \sqrt{-1}\ast \nabla_A u|^{2} + \frac{1}{2}\big |\ep  \ast \sqrt{-1}F_{A} - \frac{1 - |u|^{2}}{2\ep}\big|^{2} \nonumber\\
&+ \int_{\Sigma} \frac{1}{4}|\nabla_A u + \sqrt{-1}\ast \nabla_A u|^{2} +\frac{1}{2} \big |\ep  \ast \sqrt{-1}F_{A} + \frac{1 - |u|^{2}}{2\ep}\big|^{2} \nonumber \\
=\ & \int_{\Sigma} \frac{1}{2}|\nabla_A u \mp \sqrt{-1}\ast \nabla_A u|^{2} + \big |\ep  \ast \sqrt{-1}F_{A} \mp \frac{1 - |u|^{2}}{2\ep}\big|^{2} \pm \int_{\Sigma} \sqrt{-1}F_{A},
\end{align}
where $\ast$ denotes the Hodge star operator on $\Sigma$, and the signs are chosen so that the term $\pm\int_{\Sigma}\sqrt{-1}F_{A}$ is nonnegative. Note that if $\Sigma$ is closed, then of course
\begin{equation}\label{eq:deg-vortex}
\frac{1}{2\pi}\int_{\Sigma}\sqrt{-1}F_{A} = \deg L.
\end{equation}
Also, if $\Sigma = \RR^{2}$, then by~\cite[Proposition II.3.5]{JT} and~\cite{Ai}, under the assumption that $E_{\ep}(u, \nabla_{A})$ is finite, we still have $\int_{\Sigma}\sqrt{-1}F_{A}  \in 2\pi\ZZ$, the integer being essentially the degree of $u$ at infinity. 

Thus, ~\eqref{eq:Bogol} provides a lower bound for $E_{\ep}(u, \nabla_{A})$ in terms of a topological quantity, namely $\frac{1}{2\pi}\int_{\Sigma}\sqrt{-1}F_{A}$, also known as the vortex number. Configurations $(u, \nabla_{A})$ which attain this bound satisfy, depending on the sign of the vortex number, one of the following two first-order systems, which we will refer to collectively as the vortex equations. 
\begin{equation}\label{eq:vortex}
\left\{
\begin{array}{cl}
\nabla_A u &=  \sqrt{-1}\ast \nabla_A u, \\
\ast \sqrt{-1}F_{A} &= \frac{1 - |u|^{2}}{2\ep^{2}}.
\end{array}
\right.
\end{equation}
\begin{equation}\label{eq:anti-vortex}
\left\{
\begin{array}{cl}
\nabla_A u &= - \sqrt{-1}\ast \nabla_A u, \\
\ast \sqrt{-1}F_{A} &=- \frac{1 - |u|^{2}}{2\ep^{2}}.
\end{array}
\right.
\end{equation}
Since solutions to~\eqref{eq:vortex} or~\eqref{eq:anti-vortex} minimize $E_{\ep}$ among configurations with the same vortex number, they are, in particular, stable solutions to~\eqref{eq:YMH} when $\Sigma$ is closed. Here by a stable solution we mean a solution at which the second variation of $E_{\ep}$ is positive semi-definite. (See Section~\ref{sec:YMH-basics}.) In view of this property of vortex solutions, it seems natural to ask whether the converse is also true; that is, given a stable solution $(u, \nabla_A)$ to~\eqref{eq:YMH} on a closed $\Sigma$, must it satisfy one of~\eqref{eq:vortex} and~\eqref{eq:anti-vortex}? Our main result gives a positive answer when $\Sigma$ is the round $S^2$ or flat $T^2$, provided $u$ is not the zero section. In the $S^2$ case, this last assumption can be dropped if $\ep$ is below a threshold that depends only on $\deg L$.

The reader may wonder if the vortex equations actually admit any solutions at all. Thus, we briefly digress to recall some fundamental existence and classification results for solutions to~\eqref{eq:vortex} and~\eqref{eq:anti-vortex}. When $\Sigma = \RR^2$, these are due to Taubes~\cite{Tau2}, who showed that, up to gauge equivalence, solutions with vortex number $d$ are in one-to-one correspondence with unordered $|d|$-tuples of points on $\RR^2$. On the other hand, if $\Sigma$ is a closed surface, a similar classification was established, using different methods, by Bradlow~\cite{Br1} and Garc\'ia-Prada~\cite{GP1} (see also Noguchi~\cite{No}), under the assumption that $4\pi |\deg L| < \ep^{-2}|\Sigma|$, with the case of equality addressed in~\cite[Theorem 4.7]{Br1}. (In fact they studied a slightly different equation, but the analysis is essentially the same.) Note that the condition 
\begin{equation}\label{eq:ep-threshold}
4\pi |\deg L| \leq \ep^{-2} |\Sigma|
\end{equation}
is necessary for either~\eqref{eq:vortex} or~\eqref{eq:anti-vortex} to admit a solution, as can be seen by integrating the second lines of~\eqref{eq:vortex} or~\eqref{eq:anti-vortex} over $\Sigma$ (\cite{Br1, GP1}).

We now return to the relation between stable solutions to~\eqref{eq:YMH} and solutions to the vortex equations, and state our main theorem.
\begin{thm}\label{thm:main} 
Let $\Sigma$ be the round $S^{2}$ or a flat $T^2$ and suppose $L$ is a Hermitian line bundle over $\Sigma$ with $\deg L = d$. Let $(u, \nabla_A)$ be a stable weak solution to~\eqref{eq:YMH} and assume either of the following two conditions: 
\begin{enumerate}
\item[(H)] $u$ is not identically zero.
\vskip 1mm
\item[(H')] $\Sigma$ is the round $S^{2}$ and $|d| \leq \ep^{-2}$.
\end{enumerate} 
Then $(u, \nabla_A)$ satisfies~\eqref{eq:vortex} or~\eqref{eq:anti-vortex}.
\end{thm}
\begin{rmk}\label{rmk:main}
\begin{enumerate}
\item[(1)]
As opposed to the case $\Sigma = \RR^2$, where Taubes~\cite{Tau1} showed that \textit{any} finite-action solution to~\eqref{eq:YMH} in fact satisfies one of the vortex equations, Theorem~\ref{thm:main} with the word ``stable'' removed is false in general. For instance, if $L$ is trivial and $\ep$ is small enough, then the min-max construction of Pigati--Stern~\cite[Section 7.1]{PiSt} produces solutions $(u, \nabla_A)$ with $E_{\ep}(u, \nabla_A) > 0$ and $u \not\equiv 0$. (See equation (7.1) in~\cite{PiSt}.) These cannot satisfy either of the vortex equations, for otherwise we would get the following contradiction
\[
0 < E_{\ep}(u, \nabla_A) = \pm\int_{\Sigma}F_A = 0.
\]
Here in the middle equality we use~\eqref{eq:Bogol} and the right-most equality follows since here the bundle is trivial. 
\vskip 1mm
\item[(2)] The assumption $|d| \leq \ep^{-2}$ in (H') is just~\eqref{eq:ep-threshold} since $|S^2| = 4\pi$. This threshold for $\ep$ is optimal in that if $|d| > \ep^{-2}$, then the solution $(0, \nabla_A)$ with $F_A$ harmonic is stable but does not satisfy either~\eqref{eq:vortex} or~\eqref{eq:anti-vortex}. We elaborate on this at the end of Section~\ref{sec:stable-smooth}.
\end{enumerate}
\end{rmk}

For the remainder of this introduction, we will attempt to put Theorem~\ref{thm:main} into context, before briefly describing the idea of its proof. First, when $\Sigma$ is a convex domain in $\RR^2$, in which case $L$ is necessarily trivial, Jimbo--Sternberg~\cite{JiSt} established the constancy of stable solutions to~\eqref{eq:YMH} under a natural variational boundary condition, and for a more general class of potentials. For the Ginzburg--Landau equation $\ep^2\Delta u = (1 - |u|^2)u$ on complex-valued functions, a similar result was proved by Jimbo--Morita~\cite{JM}, assuming the homogeneous Neumann condition. Other related results on the classification of stable solutions to equations similar to~\eqref{eq:YMH} can be found, for instance, in~\cite{CaHo, Mat, Se, Ch}.

Next, recall that several other functionals in differential geometry admit special minimizers given by first-order equations. For instance, $\pm$-holomorphic maps are homotopy minimizers for the Dirichlet energy of maps between compact K\"ahler manifolds, and connections with self-dual or anti-self-dual curvatures minimize the Yang--Mills functional on compact $4$-manifolds. The relationship between complex subvarieties of K\"ahler manifolds and the area functional also falls into this framework, thanks to the Wirtinger inequality. In all these settings, ``stability $\Rightarrow$ first-order reduction'' results analogous Theorem~\ref{thm:main} have been obtained under suitable assumptions. See for instance~\cite{SY, BBDP} (harmonic maps),~\cite{BL, St} (Yang--Mills connections),~\cite{LS, Mi} (minimal submanifolds). 

Finally, we mention that the vortex equations admits various important generalizations; see for instance the surveys~\cite{GPsurvey}, ~\cite{BrGP}, and the references therein. We are particularly interested in the case of a Hermitian line bundle over a compact K\"ahler manifold $M$~\cite{Br1, GP2}. Here the solutions are again stable critical points of $E_{\ep}$, and are essentially in one-to-one correspondence with codimension-one complex subvarieties of $M$. In view of the recent work of Pigati--Stern~\cite{PiSt}, which revealed a close relationship between solutions to~\eqref{eq:YMH} on a Riemannian manifold and (real) codimension-two minimal submanifolds, it will be interesting to see whether a statement like Theorem~\ref{thm:main} holds when $M = \CC\PP^n$. The result, if true, would be an analogue of the classical theorem of Lawson--Simons~\cite{LS}, which reduces stable stationary integral currents in $\CC\PP^n$ to complex subvarieties, and would serve as further evidence for the link between $E_{\ep}$ and the volume functional in codimension two. We hope to address this question in a future work.

\vskip 2mm
\noindent\textbf{Method.} Here we assume $\Sigma$ is as in the statement of Theorem~\ref{thm:main}. The proof of Theorem~\ref{thm:main} shares a common theme with a lot of the results cited above, particularly the work of Bourguignon--Lawson~\cite[Section 10]{BL} on stable $SU(2)$ Yang--Mills connections on homogeneous four-manifolds. To describe the idea in our setting, we consider the one-parameter group of diffeomorphisms generated by a vector field $X$ on $\Sigma$. Pulling back a solution $(u, \nabla_A)$ via these diffeomorphisms yields a one-parameter family of configurations, along which we compute the second derivative of $E_{\ep}$. This has the same effect as computing the second variation of $E_{\ep}$ along the path $(u + t (\nabla_A u)_{X}, \nabla_A + t(\iota_{X}F_A))$ (see~\cite[pp. 198--199]{BL}), which has to be non-negative by the stability assumption. Of course, as $E_{\ep}$ is isometry-invariant, choosing $X$ to be a Killing field yields no useful information since we would get zero anyway, regardless of stability. However, information can be extracted if we keep $X$ Killing but replace $(u + t (\nabla_A u)_{X}, \nabla_A + t(\iota_{X}F_A))$ by $(u + t \sigma_X, \nabla_A - t\sqrt{-1}\iota_{X}\varphi)$, where 
\[
\sigma = \nabla_A u - \sqrt{-1}\ast \nabla_A u,\ \varphi =  \sqrt{-1}F_A - \ast\frac{1 - |u|^{2}}{2\ep^{2}}.
\]
This choice is inspired by the one in~\cite[Section 10]{BL}. Note that, a priori, the second variation of $E_{\ep}$ in this direction does not have to be zero, but if in addition $(u, \nabla_A)$ verifies~\eqref{eq:vortex} or~\eqref{eq:anti-vortex}, then $(\sigma_X, \iota_X \varphi) =  (0, 0)$  or $(2(\nabla_A u)_X, 2\sqrt{-1}\iota_X F_A)$, and the second variation vanishes in either case as $X$ is Killing. Thus one expects $\sigma$ and $\varphi$ to be helpful in detecting solutions to the vortex equations. 

Computing the second variations of $E_{\ep}$ along $(u + t \sigma_X, \nabla_A - t\sqrt{-1}\iota_{X}\varphi)$ gives rise to a quadratic form $Q$ defined over the space $\cK$ of Killing fields on $\Sigma$, which must be positive semi-definite if $(u, \nabla_A)$ is a stable solution. As in~\cite{BL}, the proof then boils down to taking the trace over $\cK$, and observing that when $\Sigma$ is as in the statement of Theorem~\ref{thm:main}, the resulting inequalities, together with some basic estimates for solutions of~\eqref{eq:YMH}, allow us to conclude the proof assuming (H). The prove the Theorem assuming (H') instead, we first observe that when $|d| = \ep^{-2}$, the conclusion holds when even if $u \equiv 0$. Then we argue that $u \equiv 0$ contradicts stability when $|d| < \ep^{-2}$, thanks to an estimate on the lowest eigenvalue of $d_A^{\ast}d_A$ due to Kuwabara~\cite{Ku}.
\vskip 2mm
\noindent\textbf{Notation.} For the rest of the paper, $(\Sigma, g)$ will be a closed oriented surface equipped with a Riemannian metric, and $L$ a Hermitian line bundle over $\Sigma$. The Levi--Civita connection on $\Sigma$ is denoted by $\nabla$, and the volume form by $d\mu_g$. The curvature convention we adopt is
\[
R_{X, Y}Z = \nabla^{2}_{X, Y}Z - \nabla^{2}_{Y, X}Z.
\]

We use the same pointed brackets $\langle \cdot, \cdot\rangle$ to denote any bundle metric on $\Lambda^{p}T^{\ast}\Sigma$ or $\Lambda^{p}T^{\ast}\Sigma \otimes L$ that is induced by $g$ and the Hermitian metric on $L$. As usual, we denote by $\Omega^{p}(\Sigma)$ the space of $p$-forms on $\Sigma$, and by $\Omega^{p}(L)$ the space of sections of $\Lambda^{p}T^{\ast}\Sigma \otimes L$. Integrating the bundle metrics over $\Sigma$ yields inner products on $\Omega^{p}(\Sigma)$ and $\Omega^{p}(L)$.

Fixing once an for all a smooth background metric connection $\nabla_0$ on $L$, any other metric connection can be written as $\nabla_{A} := \nabla_0 - \sqrt{-1}A$, where $A$ is a real $1$-form on $\Sigma$. The curvatures of $\nabla_A$ and $\nabla_0$ are then related by $F_A = F_0 - \sqrt{-1}dA$. The exterior derivative induced by $\nabla_{A}$ on $\Omega^{\ast}(L)$ is denoted $d_{A}$, and its formal adjoint $d_A^{\ast}$. Similarly, $\nabla_{A}^{\ast}$ denotes the adjoint of $\nabla_{A}$. For instance, for $\sigma \in \Omega^1(L)$, we have
\[
\nabla_{A}^{\ast}\sigma = -(\nabla_{A}\sigma)_{e_i, e_i}, 
\]
where the right-hand side is summed over an orthonormal basis $e_1, e_2$ of $T_{x}\Sigma$ at each $x \in \Sigma$. (Below, unless otherwise stated, repeated indices are always summed.) Also, for a section $u$ of $L$, we will use $d_A u$ and $\nabla_{A}u$ interchangeably.  

Next, by $\Delta$ we will always mean the Hodge Laplacian $d^{\ast}d + d d^{\ast}$, even when it is acting on $\Omega^{0}(\Sigma)$. Thus, for example, in this notation a real function $f$ is sub-harmonic if $\Delta f \leq 0$. Similarly, using $d_A$ and $d_{A}^{\ast}$, the Hodge Laplacian acting on $\Omega^{\ast}(L)$ is given by 
\[
\Delta_{A} = d_{A} d_{A}^{\ast} + d_{A}^{\ast}d_{A}.
\]
When there is no danger of confusion, we will sometimes drop the subscripts in $\nabla_{A}$, $\Delta_A$, $d_{A}$, $d^{\ast}_{A}$, etc. and simply write them as $\nabla, \Delta, d, d^{\ast}$, etc.

Finally, by a configuration we mean a pair $(u, \nabla_A)$ where $u$ is a section of $L$ and $\nabla_A$ is a metric connection on $L$, with regularity to be specified depending on the context. Given $\ep > 0$ and a configuration $(u, \nabla_A)$, we define
\begin{align*}
h_{(u, \nabla_A)} &= \frac{1 - |u|^{2}}{2\ep^{2}}\ (\text{$|u|$ is computed using the bundle metric on $L$}),\\
f_{(u, \nabla_A)} &= \ast \sqrt{-1}F_{A},\\
\sigma_{(u, \nabla_A)} &= \nabla_A u - \sqrt{-1}\ast \nabla_A u,\\
\varphi_{(u, \nabla_A)} &= \sqrt{-1}F_A - \ast h = \ast (f - h).
\end{align*}
The subscripts $(u, \nabla_A)$ will be dropped when it's clear from the context which configuration we mean. Also, note that $h, f, \varphi$ stay unchanged when we switch from $(u, \nabla_A)$ to a gauge equivalent configuration $(e^{\sqrt{-1}\theta}u, \nabla_A - \sqrt{-1}d\theta)$, whereas $\sigma$ transforms by
\[
\sigma_{(e^{\sqrt{-1}\theta}u, \nabla_A - \sqrt{-1}d\theta)} = e^{\sqrt{-1}\theta}\sigma_{(u, \nabla_A)}.
\]
Nonetheless, $|u|, |\sigma|$ and $\langle u, \sigma \rangle$ are still gauge invariant. Other notation and terminology will be introduced when needed.
\vskip 2mm
\noindent\textbf{Organization.} In Section~\ref{sec:weitz} we review a couple of Weitzenb\"ock-type formulas and note some consequences which are important for the computations to follow. Section~\ref{sec:YMH-basics} collects a number of basic facts about $E_{\ep}$ and~\eqref{eq:YMH}, including the first and second variation formula, regularity of weak solutions up to change of gauge, and some basic pointwise estimates which help us distinguish vortices from other solutions of~\eqref{eq:YMH}. At the end we also recall how to derive~\eqref{eq:Bogol}. In Section~\ref{sec:stable-smooth}, we prove Theorem~\ref{thm:main} and elaborate on Remark~\ref{rmk:main}(2). 

\vskip 2mm
\noindent\textbf{Acknowledgments.} I would like to thank Andr\'e Neves and Guangbo Xu for helpful conversations related to this work.

\section{Review of Weitzenb\"ock formulas and some consequences}\label{sec:weitz}
Both of the Weitzenb\"ock formulas recalled below are standard and the proofs can be found essentially in~\cite[Section 3]{BL}. Note that because $U(1)$ is abelian, the formulas simplify somewhat in our case. 
\begin{prop}\label{prop:weitzenbock} Let $\sigma \in \Omega^{1}(L)$ and $\varphi \in \Omega^{2}(\Sigma)$. Then the following hold.
\begin{enumerate}
\item[(a)] (See also~\cite[Theorem 3.2]{BL})$\\
(\Delta_{A}\sigma)_{X} = (\nabla_{A}^{\ast}\nabla_{A} \sigma)_{X} + (F_{A})_{e_{i}, X}\sigma_{e_{i}} + \sigma_{\Ric(X)}$, where the second term on the right-hand side is summed over an orthonormal basis $\{e_{i}\}$ of $T_{x}\Sigma$.
\vskip 1mm
\item[(b)] (\cite[Theorem 3.10]{BL})\\
$(\Delta \varphi)_{X, Y} = (\nabla^{\ast}\nabla\varphi)_{X, Y} + \varphi_{\Ric(X), Y} + \varphi_{X, \Ric(Y)} + \varphi_{e_{i}, R_{X, Y}e_{i}}$, where again the last term is summed over an orthonormal basis, and $R_{X, Y}$ denotes the curvature tensor on the base manifold $\Sigma$.
\end{enumerate}
\end{prop}
Besides the Weitzenb\"ock formulas, we also need to know how differential operators like $d, d^{\ast}$ and $d_{A}, d_{A}^{\ast}$ interact with the operation of contracting with Killing vector fields on $\Sigma$. We don't think these formulas are new, but we still include their proofs for the reader's convenience.
\begin{lemm}\label{lemm:contraction} Let $\sigma$ and $\varphi$ be as in Proposition~\ref{prop:weitzenbock} and suppose $X$ is a Killing vector field on $\Sigma$. Then the following hold.
\begin{enumerate}
\item[(a)] $d_{A}^{\ast}d_{A}(\sigma_{X}) = (\Delta_{A}\sigma)_{X} - (F_{A})_{e_{i}, X}\sigma_{e_{i}} - \langle dX^{\flat}, d_A\sigma \rangle$, where $X^{\flat}$ is the $1$-form dual to $X$, and $\langle dX^{\flat}, d_A\sigma \rangle \in \Omega^0(L)$ is given by
\[
\langle dX^{\flat}, d_A\sigma \rangle = 2\sum_{i < j}\langle \nabla_{e_i}X, e_j \rangle (d_A\sigma)_{e_i, e_j}.
\]
\vskip 1mm
\item[(b)] $d^{\ast} (\iota_{X}\varphi) = -(d^{\ast}\varphi)_{X} + \langle dX^{\flat}, \varphi \rangle$.
\vskip 1mm
\item[(c)] $(\Delta \iota_{X}\varphi)_{e_j} = (\Delta \varphi)_{X, e_{j}} - \varphi_{e_{i}, R_{X, e_{j}}e_i} - 2(\nabla\varphi)_{e_i, \nabla_{e_i}X, e_j}$.
\vskip 1mm
\item[(d)] $d^{\ast}d(\iota_{X}\varphi)_{e_{j}} = (\Delta \varphi)_{X, e_{j}} + (d\iota_{X}d^{\ast}\varphi)_{e_j}$.
\end{enumerate}
\end{lemm}
\begin{proof}
Fix a point $p \in \Sigma$ and let $\{e_{i}\}$ be a local orthonormal frame near $p$ with $\nabla e_{i} = 0$ at $p$ for all $i$. To see (a), we start by computing (below we drop the subscripts in $\nabla_A, d_{A}$, etc.)
\begin{align*}
d^{\ast}d(\sigma_{X}) &= -\nabla_{e_{i}}\nabla_{e_{i}}(\sigma_X) = -\nabla_{e_{i}} \big( (\nabla\sigma)_{e_{i}, X} + \sigma_{\nabla_{e_i}X}\big)\\
&= -(\nabla^{2}\sigma)_{e_i, e_i, X} - 2(\nabla\sigma)_{e_i, \nabla_{e_i}X} - \sigma_{\nabla^{2}_{e_i, e_i}X}\\
&= (\nabla^{\ast}\nabla\sigma)_{X} - 2\langle  \nabla_{e_i}X, e_j \rangle (\nabla\sigma)_{e_i, e_j} - \sigma_{R_{e_i, X}e_i}.
\end{align*}
In getting the last line we wrote $\nabla_{e_i}X = \langle \nabla_{e_i}X, e_j \rangle e_j$ and also used the fact that, when $X$ is a Killing vector field, we have
\begin{equation}\label{eq:Killing-curvature}
\nabla^{2}_{V, W}X = R_{V, X}W.
\end{equation}
(The identity~\eqref{eq:Killing-curvature} be frequently used in what follows, sometimes without further comment.) To continue, we use Proposition~\ref{prop:weitzenbock}(a) to replace $\nabla^{\ast}\nabla\sigma$ and also note that $\sigma_{R_{e_i, X}e_i} = -\sigma_{\Ric(X)}$. Then we obtain
\begin{align*}
d^{\ast}d(\sigma_{X}) = (\Delta\sigma)_{X} - F_{e_i, X}\sigma_{e_i}  - 2\langle  \nabla_{e_i}X, e_j \rangle (\nabla\sigma)_{e_i, e_j} .
\end{align*}
We get (a) upon noting that since $X$ is a Killing vector field, $\langle \nabla_{V}X, W \rangle$ is skew-symmetric in $V, W$, and hence we have
\begin{align*}
2\sum_{i, j}\langle  \nabla_{e_i}X, e_j \rangle (\nabla\sigma)_{e_i, e_j} &= 2\sum_{i < j}\langle \nabla_{e_i}X, e_j \rangle \big( (\nabla\sigma)_{e_i, e_j} - (\nabla\sigma)_{e_j, e_i} \big)=\\
&= 2\sum_{i < j}\langle \nabla_{e_i}X, e_j \rangle (d\sigma)_{e_i, e_j} = \langle dX^{\flat}, d\sigma \rangle.
\end{align*}
Of course in getting the last equality we used the fact that, at $p$, 
\[
(dX^{\flat})_{e_i, e_j} = \nabla_{e_i}(X^{\flat}(e_j)) - \nabla_{e_j}(X^{\flat}(e_i)) = 2\langle \nabla_{e_i}X, e_j \rangle.
\]

To see (b), we compute 
\begin{align*}
d^{\ast}(\iota_{X}\varphi) &= -\nabla_{e_i}(\varphi_{X, e_{i}})= -(\nabla\varphi)_{e_i, X, e_i} - \varphi_{\nabla_{e_i}X, e_i}\\
&= (\nabla\varphi)_{e_i, e_i, X} + \sum_{i, j}\langle \nabla_{e_i}X, e_j \rangle \varphi_{e_i, e_j}\\
&= -(d^{\ast} \varphi)_{X} + 2\sum_{i < j}\langle \nabla_{e_i}X, e_j \rangle \varphi_{e_i, e_j}\\
&= -(d^{\ast}\varphi)_{X} + \langle dX^{\flat}, \varphi \rangle.
\end{align*}

To see (c), we start by computing
\begin{align*}
(d^{\ast} d \iota_{X}\varphi)_{e_j} =& -\nabla_{e_i}\big( (d\iota_{X}\varphi)_{e_i, e_j} \big) = -\nabla_{e_i}\big( (\nabla\iota_X\varphi)_{e_i, e_j} - (\nabla\iota_{X}\varphi)_{e_j, e_i} \big)\\
=& -\nabla_{e_i}\big( (\nabla\varphi)_{e_i, X, e_j} + \varphi_{\nabla_{e_i}X, e_j}  - (\nabla\varphi)_{e_j, X, e_i} - \varphi_{\nabla_{e_j}X, e_i} \big)\\
=& (\nabla^{\ast}\nabla\varphi)_{X, e_j} - 2(\nabla\varphi)_{e_i, \nabla_{e_i}X, e_j} - \varphi_{\nabla^2_{e_i, e_i}X, e_j}\\
&+ (\nabla^{2}\varphi)_{e_i, e_j, X, e_i}+ (\nabla\varphi)_{e_j, \nabla_{e_i}X, e_i} + (\nabla\varphi)_{e_i, \nabla_{e_j}X, e_i} + \varphi_{\nabla^{2}_{e_i, e_j}X, e_i}\\
=& (\nabla^{\ast}\nabla\varphi)_{X, e_j} - 2(\nabla\varphi)_{e_i, \nabla_{e_i}X, e_j} + \varphi_{\Ric(X), e_j}\\
&+ (\nabla^{2}\varphi)_{e_i, e_j, X, e_i}+ (\nabla\varphi)_{e_j, \nabla_{e_i}X, e_i} + (\nabla\varphi)_{e_i, \nabla_{e_j}X, e_i} + \varphi_{\nabla^2_{e_i, e_j}X, e_i}.
\end{align*}
Here, in getting the last line we used again~\eqref{eq:Killing-curvature} to replace $\nabla^2_{e_i, e_i}X$ by $R_{e_i, X}e_i = -\Ric(X)$. Next, we have
\begin{align*}
(dd^{\ast}\iota_{X}\varphi)_{e_j}=& \nabla_{e_j}\big( -(\nabla\iota_{X}\varphi)_{e_i, e_i} \big)\\
=& -\nabla_{e_j}\big( (\nabla\varphi)_{e_i, X, e_i} + \varphi_{\nabla_{e_i}X, e_i} \big)\\
=& -(\nabla^{2}\varphi)_{e_j, e_i, X, e_i} - (\nabla\varphi)_{e_i, \nabla_{e_j}X, e_i} - (\nabla\varphi)_{e_j, \nabla_{e_i}X, e_i} - \varphi_{\nabla^{2}_{e_j, e_i}X, e_i}.
\end{align*}
Adding up the previous two computations, we observe that all terms involving $(\nabla\varphi)$ cancel, except for the term $-2(\nabla\varphi)_{e_i, \nabla_{e_i}X, e_j}$. Thus we get
\begin{align*}
(\Delta\iota_{X}\varphi)_{e_j} =& (\nabla^{\ast}\nabla\varphi)_{X, e_j} - 2(\nabla\varphi)_{e_i, \nabla_{e_i}X, e_j} + \varphi_{\Ric(X), e_j}\\
& + (\nabla^{2}\varphi)_{e_i, e_j, X, e_i}-(\nabla^{2}\varphi)_{e_j, e_i, X, e_i} + \varphi_{\nabla^{2}_{e_i, e_j}X, e_i} - \varphi_{\nabla^{2}_{e_j, e_i}X, e_i}\\
=& (\nabla^{\ast}\nabla\varphi)_{X, e_j} - 2(\nabla\varphi)_{e_i, \nabla_{e_i}X, e_j} + \varphi_{\Ric(X), e_j}\\
&- \varphi_{R_{e_i, e_j}X, e_i} - \varphi_{X, R_{e_i, e_j}e_i}+ \varphi_{R_{e_i, e_j}X, e_i}\\
=& (\nabla^{\ast}\nabla\varphi)_{X, e_j} - 2(\nabla\varphi)_{e_i, \nabla_{e_i}X, e_j} + \varphi_{\Ric(X), e_j}+ \varphi_{X, \Ric(e_j)}.
\end{align*}
We obtain exactly the formula in (c) upon replacing $(\nabla^{\ast}\nabla\varphi)_{X, e_j}$ using Proposition~\ref{prop:weitzenbock}(b) and cancelling the terms involving the Ricci curvature.

Finally, to see (d), recalling from the proof of (b) that 
\[
\langle dX^{\flat}, \varphi \rangle = \langle \nabla_{e_i}X, e_j \rangle \varphi_{e_i, e_j}= \varphi_{e_i,\nabla_{e_i}X},
\]
we may use (b) to compute
\begin{align*}
-dd^{\ast}(\iota_{X}\varphi)_{e_j} =& \nabla_{e_j}\big( (d^{\ast}\varphi)_{X} - \langle dX^{\flat}, \varphi \rangle \big)\\
=& (d\iota_{X}d^{\ast}\varphi)_{e_j} - \nabla_{e_j}(\varphi_{e_i, \nabla_{e_i}X})\\
=& (d\iota_{X}d^{\ast}\varphi)_{e_j} - (\nabla\varphi)_{e_j, e_i, \nabla_{e_i}X} - \varphi_{e_i, R_{e_j, X}e_i}\\
=& (d\iota_{X}d^{\ast}\varphi)_{e_j} - (d\varphi)_{e_j, e_i, \nabla_{e_i}X} + (\nabla\varphi)_{e_i, \nabla_{e_i}X, e_j}\\& + (\nabla\varphi)_{\nabla_{e_i}X, e_j, e_i} + \varphi_{e_i, R_{X, e_j}e_i}.
\end{align*}
To continue, we use the anti-symmetry of $\langle\nabla_{\cdot}X, \cdot\rangle$ to compute
\begin{align*}
(\nabla\varphi)_{e_i, \nabla_{e_i}X, e_j}+ (\nabla\varphi)_{\nabla_{e_i}X, e_j, e_i} =& \langle \nabla_{e_i}X,e_k \rangle \big( (\nabla\varphi)_{e_i, e_k, e_j} - (\nabla\varphi)_{e_k, e_i, e_j} \big)\\
=& 2\langle \nabla_{e_i}X, e_k \rangle(\nabla\varphi)_{e_i, e_k, e_j}\\
=& 2(\nabla\varphi)_{e_i, \nabla_{e_i}, e_j}.
\end{align*}
In summary we've obtained
\[
-dd^{\ast}(\iota_{X}\varphi)_{e_j} = (d\iota_{X}d^{\ast}\varphi)_{e_j} + (d\varphi)_{e_i, e_j, \nabla_{e_i}X} + 2(\nabla\varphi)_{e_i, \nabla_{e_i}X, e_j} + \varphi_{e_i, R_{X, e_j}e_i}.
\]
We complete the proof of (d) upon adding this to (c) and noting that $d\varphi = 0$ since $\Sigma$ is two-dimensional.
\end{proof}

\section{Review of some basic facts about the abelian Yang--Mills--Higgs and vortex equations}\label{sec:YMH-basics}
We begin by reviewing the first and second variation formulas of $E_{\ep}$. Let $\cC$ to be the set of configurations $(u, \nabla_A)$ where $u$ is a section of $L$ of class $L^{\infty} \cap W^{1, 2}$ and $\nabla_A = \nabla_0 - \sqrt{-1}A$ is a metric connection of class $W^{1, 2}$. (Recall that $\nabla_0$ is our fixed reference connection on $L$.) The latter means that the real $1$-form $A$ lies in $W^{1, 2}$. Note that we then have $F_A = F_{0} - \sqrt{-1}dA$.

Given $(u, \nabla_A) \in \cC$, and a pair $(v, a) \in \Omega^{0}(L) \times \Omega^{1}(\Sigma)$, recall that the first variation of $E_{\ep}$ is given by
\begin{align}
\delta E_{\ep}(u, \nabla_A)(v, a) :=& \frac{d}{dt}E_{\ep}(u + tv, \nabla_A - t\sqrt{-1}a)\nonumber\\
=& \int_{\Sigma} 2\ep^{2}\langle \sqrt{-1}F_A, da \rangle + 2\re\langle \nabla_A u, \nabla_A v  - \sqrt{-1}au \rangle + \frac{|u|^{2} - 1}{\ep^{2}}\re\langle u, v \rangle d\mu_{g} \label{eq:1st-variation}
\end{align}
Of course, $(u, \nabla_A)$ is a weak solution to~\eqref{eq:YMH} if and only if $\delta E_{\ep}(u, \nabla_A) = 0$. Moreover, any weak solution to~\eqref{eq:YMH} are locally gauge equivalent to a smooth solution. (See Proposition~\ref{prop:regularity} below for a more precise statement.) Next we recall the second variation formula for $E_{\ep}$, which for instance may be found in~\cite{GuSi}:
\begin{align}
\delta^{2} E_{\ep}(u, \nabla_A)(v, a) := &\frac{d^2}{dt^2}E_{\ep}(u + tv, \nabla_A - t\sqrt{-1}a)\nonumber\\
=& \int_{\Sigma} 2\ep^{2}|da|^{2} + 2|d_A v|^{2} - 4\langle a, u\times d_A v + v \times d_A u \rangle + 2|u|^{2}|a|^{2} \nonumber\\
&+ \frac{|u|^{2} - 1}{\ep^{2}}|v|^{2} + \frac{2(\re\langle u, v \rangle)^{2}}{\ep^{2}}d\mu_g \label{eq:2nd-variation}.
\end{align}
Here the cross product ``$\times$'' has the following meaning
\[
\xi \times \eta = \re\langle \sqrt{-1}\xi, \eta \rangle, \text{ for }\xi, \eta \in L_{x} \text{ and for all }x \in \Sigma.
\]
The following fact justifies the term ``cross product'':
\begin{equation}\label{eq:cross-product}
\re\langle  \sqrt{-1}\xi, \eta \rangle = -\re\langle \xi, \sqrt{-1}\eta \rangle = -\re\langle \sqrt{-1}\eta, \xi \rangle.
\end{equation}

Polarizing~\eqref{eq:2nd-variation} and then formally integrating by parts, we get the following Jacobi operators $J^{1}_{(u, \nabla_A)}, J^{2}_{(u, \nabla_A)}$, which already appeared for example in~\cite[Section 3]{GuSi}:
\begin{align}
J^{1}_{(u, \nabla_A)}(v, a) &:= d_{A}^{\ast}d_{A}v + 2\langle  a, \sqrt{-1}d_{A} u\rangle - (d^{\ast}a)\sqrt{-1}u + \frac{|u|^{2} - 1}{2\ep^{2}}v + \frac{\re\langle u, v \rangle}{\ep^{2}}u. \label{eq:1st-Jacobi}\\
J^{2}_{(u, \nabla_A)}(v, a) &:= \ep^{2}d^{\ast}da - u \times d_{A} v  - v \times d_{A}u + |u|^{2}a. \label{eq:2nd-Jacobi}
\end{align}
These have the property that when $u, \nabla, v$ and $a$ are sufficiently regular, we have
\[
\delta^2 E_{\ep}(u, \nabla_A)(v,a) = 2\int_{\Sigma}  \re\langle J^{1}_{(u, \nabla_A)}(v, a), v \rangle + \langle J^{2}_{(u, \nabla_A)}(v, a), a \rangle d\mu_g.
\]
\begin{defi}\label{defi:stable}
A solution $(u, \nabla_A) \in \cC$ to~\eqref{eq:YMH} is said to be \textit{stable} if $\delta^{2}E_{\ep}(u, \nabla_A)(v, a)$ as defined in~\eqref{eq:2nd-variation} is non-negative for any $(v, a) \in \Omega^{0}(L) \times \Omega^{1}(\Sigma)$. 
\end{defi}
\begin{rmk}\label{rmk:singular-variation}
Note that if $(u, \nabla_A)$ is stable, then in fact $\delta^{2}E_{\ep}(u, \nabla_A)(v, a)\geq 0$ even if $v$ is merely a section of $L$ of class $W^{1, 2} \cap L^{\infty}$, and $a$ is a $1$-form of class $W^{1,2}$. This follows by inspecting the integrands in~\eqref{eq:2nd-variation} and noting that $v$ and $a$ can be smoothly approximated, respectively, in the $W^{1,2} \cap L^{p}$ and $W^{1, 2}$ topology. (Here $p < \infty$ is arbitrary.) 
\end{rmk}

Next we give a more precise statement of the regularity of weak solutions mentioned before. The result is due to Taubes~\cite{Tau1}. Let $U \subset \Sigma$ be an open set, and assume that $L$ has a local, non-vanishing section over $U$, then we have a unitary trivialization of $L\vert_{U}$, under which sections are identified with complex-valued functions, and each metric connection can be written as $\nabla = d - \sqrt{-1}B$ for some real-valued $1$-form $B$ on $U$.
\begin{prop}[\cite{Tau1}, Proposition 4.1]
\label{prop:regularity}
Let $(u, \nabla) \in \cC$ be a weak solution to~\eqref{eq:YMH} on $\Sigma$ and let $U\subset \Sigma$ be a connected open set over which $L$ can be trivialized as above. Write $\nabla = d - \sqrt{-1}B$, and let $\theta \in W^{2, 2}(\Omega; \RR)$ be the unique solution to 
\begin{equation}
\left\{
\begin{array}{rl}
-d^{\ast}d\theta &= d^{\ast}B \text{ in }U,\\
\theta &= 0 \text{ on }\partial U.
\end{array}
\right.
\end{equation}
Then $(e^{\sqrt{-1}\theta}u, \nabla - \sqrt{-1}d\theta)$ is a smooth solution to~\eqref{eq:YMH} on $U$.
\end{prop}
Based on Proposition~\ref{prop:regularity}, one can in fact show that a weak solution $(u, \nabla) \in \cC$ is gauge-equivalent to a smooth solution over all of $\Sigma$. This was pointed out to the author by the reviewer of~\cite{Che}. Specifically, we have
\begin{prop}\label{prop:global-regularity}
Let $(u, \nabla)$ be as in Proposition~\ref{prop:regularity}. There exists $\varphi \in W^{2, 2}(\Sigma; \RR)$ such that $(e^{\sqrt{-1}\varphi}u, \nabla - \sqrt{-1}d\varphi)$ is a smooth solution to~\eqref{eq:YMH} on $\Sigma$.
\end{prop}
\begin{proof}
Take any $p \in \Sigma$ and concentric geodesic balls $B_{r}(p) \subset B_{4r}(p)$. Define $U_1 = B_{4r}(p)$ and $U_2 = \Sigma \setminus \overline{B_{r}(p)}$. Then the line bundle $L$ is trivializable over both $U_1$ and $U_2$. ($U_1$ is contractible. Also, since we are on a compact surface, $U_2$ deformation retracts onto a wedge sum of finitely many circles, over which any complex line bundle is trivializable.) Hence we may apply Proposition~\ref{prop:regularity} to get $\theta_i \in W^{2, 2}(U_i; \RR)$ such that $(u_i, \nabla_i) := (e^{\sqrt{-1}\theta_i}u, \nabla - \sqrt{-1}d\theta_i)$ is smooth on $U_i$ for $i = 1, 2$. Now note that on $U_1 \cap U_2$, the $1$-form
\[
d(\theta_1 - \theta_2) = \sqrt{-1}\big[ (\nabla - \sqrt{-1}d\theta_1) - (\nabla - \sqrt{-1}d\theta_2) \big]
\]
is smooth, and hence the $W^{2,2}$-function $\theta_1 - \theta_2$, having distributional gradient that is smooth, is itself smooth on $U_1 \cap U_2$. To finish, let $\zeta \in C^{\infty}_{c}(B_{3r}(p))$ be a cut-off function which is identically $1$ on $B_{2r}(p)$, and define 
\[
\varphi = \zeta\theta_1 + (1 - \zeta) \theta_2,
\]
Then we check that $(\widetilde{u}, \widetilde{\nabla}) := (e^{\sqrt{-1}\varphi}u, \nabla - \sqrt{-1}d\varphi)$ is equal to $(u_1, \nabla_1)$ on $B_{2r}(p)$, and to $(u_2, \nabla_2)$ on $\Sigma \setminus B_{3r}(p)$. It remains to check that $(\widetilde{u}, \widetilde{\nabla})$ is smooth on $U_1 \cap U_2$. To see that, note that on $U_1 \cap U_2 = B_{4r}(p) \setminus \overline{B_{r}(p)}$, we have 
\[
(\widetilde{u}, \widetilde{\nabla}) = (e^{\sqrt{-1}\zeta (\theta_1 - \theta_2)}u_2, \nabla_2 - \sqrt{-1}d(\zeta(\theta_1 - \theta_2))),
\]
which is smooth on $U_1 \cap U_2$ since $(u_2, \nabla_2), \zeta$ and $\theta_1 - \theta_2$ all are.
\end{proof}

Thanks or Proposition~\ref{prop:global-regularity}, from now on we can just work with smooth solutions $(u, \nabla_A)$ instead of weak solutions. Below, we recall some pointwise identities which are valid for solutions of~\eqref{eq:YMH}. The proofs in the case $\Sigma = \RR^{2}$ and $\ep = 1$ are contained in~\cite[Chapter III.6]{JT}, and the general case requires no essential modification. For the reader's convenience we sketch the argument for some of the parts.
\begin{prop}\label{prop:YMH-identities}
Let $(u, \nabla_A)$ be a smooth solution to~\eqref{eq:YMH} on $\Sigma$. Then, with $h, f, \sigma$ and $\varphi$ defined as in the introduction, the following hold.
\begin{enumerate}
\item[(a)] $\Delta h = \frac{1}{\ep^{2}}|d_{A} u|^2 - \frac{1}{\ep^{2}}|u|^2 h$.
\vskip 1mm
\item[(b)] $\Delta f = \frac{1}{\ep^{2}}\ast (d_{A}u \times d_{A}u) - \frac{1}{\ep^{2}}|u|^2f$.
\vskip 1mm
\item[(c)] $\Delta_A d_{A}u = -\frac{1}{\ep^{2}}\langle d_{A}u, u \rangle u + \sqrt{-1}f (\ast d_{A}u) + h d_{A}u$.
\vskip 1mm
\item[(d)] $d^{\ast} \varphi = \frac{1}{\ep^{2}}\re\langle \sqrt{-1}u, \sigma \rangle$.
\vskip 1mm
\item[(e)] $d_{A}\sigma = -\varphi \sqrt{-1}u$.
\end{enumerate}
In (b), the $2$-form $d_{A}u \times d_{A}u$ is defined by
\[
(d_A u \times d_{A} u)_{V, W} = 2(d_{A}u)_{V}\times (d_{A}u)_{W} = 2\re\langle \sqrt{-1}(d_A u)_V, (d_A u)_W \rangle.
\]
\end{prop}
\begin{proof}
For the proofs of parts (a)(b)(c) we refer the reader to~\cite[Proposition III.6.1]{JT}. The remaining parts are immediate consequences of~\eqref{eq:YMH}, but we include the proofs for completeness. To get (d), we recall that, since $\Sigma$ is a surface, we have 
\begin{equation}\label{eq:d-star-surface}
d^{\ast} = -\ast d \ast \text{, for forms of any degree.}
\end{equation}
Now we have
\begin{align*}
d^{\ast}\varphi &= d^{\ast}\sqrt{-1}F_{A} - d^{\ast}\ast h\\
&= \frac{1}{\ep^{2}}\re\langle \sqrt{-1}u, d_{A}u \rangle + \ast d \ast \ast h\\
&= \frac{1}{\ep^{2}}\re\langle \sqrt{-1}u, d_{A}u \rangle + \ast dh\\
&= \frac{1}{\ep^{2}}\re\langle \sqrt{-1}u, d_{A}u \rangle - \ast \frac{1}{2\ep^{2}}\big( \langle u, d_A u \rangle + \langle d_A u, u \rangle \big).
\end{align*}
In getting the second line we used~\eqref{eq:YMH} and~\eqref{eq:d-star-surface}. To continue, we have
\begin{align*}
&\frac{1}{\ep^{2}}\re\langle \sqrt{-1}u, d_{A}u \rangle - \ast \frac{1}{2\ep^{2}}\big( \langle u, d_A u \rangle + \langle d_A u, u \rangle \big)\\
=& \frac{1}{2\ep^{2}}\big( \langle \sqrt{-1}u, d_{A}u \rangle + \langle d_{A}u, \sqrt{-1}u \rangle \big) - \frac{1}{2\ep^{2}}\big( \langle u, \ast d_A u \rangle + \langle \ast d_A u, u \rangle \big)\\
=& \frac{1}{2\ep^{2}}\big( \langle \sqrt{-1}u, d_{A}u \rangle + \langle d_{A}u, \sqrt{-1}u \rangle \big) - \frac{1}{2\ep^{2}}\big( \langle \sqrt{-1} u, \sqrt{-1}\ast d_A u \rangle + \langle \sqrt{-1}\ast d_A u, \sqrt{-1}u \rangle \big).
\end{align*}
The last line follows because the metric on $L$ is Hermitian. Combining the two strings of computations above, simplifying, and recalling the definition of $\sigma$, we get
\[
d^{\ast}\varphi = \frac{1}{2\ep^2}\big( \langle \sqrt{-1}u, \sigma \rangle + \langle \sigma, \sqrt{-1}u \rangle \big),
\]
which is exactly what we want to prove.

To prove (e), note that
\begin{align*}
d_{A}\sigma&= d_{A}\big( d_A u - \sqrt{-1}\ast d_A u \big)\\
&= F_A u - \sqrt{-1} d_A \ast d_A u\\
&= F_A u + \sqrt{-1}\ast d_A^{\ast}d_A u,
\end{align*}
where, as in the proof of (d), in the last line we used the fact that $d_A^{\ast} = -\ast d_A \ast$. Using~\eqref{eq:YMH}, we may continue the computation
\[
F_A u + \sqrt{-1}\ast d_A^{\ast}d_A u = F_A u + \sqrt{-1}(\ast h) u = -\sqrt{-1}(\sqrt{-1}F_A - \ast h)u,
\]
and (e) is proved.
\end{proof}
From the identities above we deduce the following properties for solutions to~\eqref{eq:YMH} which help us detect when they are solutions to the vortex equations. 
\begin{prop}
\label{prop:estimates}
Let $(u, \nabla_A)$ be a smooth solution to~\eqref{eq:YMH} on $\Sigma$ with $u \not\equiv 0$. Then 
\begin{enumerate}
\item[(a)] (See also~\cite[Lemma III.8.4]{JT}) We have $\pm f \leq h$. Moreover, if equality is achieved at some point on $\Sigma$ then the two sides are identically equal on $\Sigma$.
\vskip 1mm
\item[(b)] (See also~\cite[p.97]{JT}) $f = h$ ($f = -h$, resp.) if and only if $d_A u = \sqrt{-1}\ast d_A u$ ($d_A u = -\sqrt{-1}\ast d_A u$, resp.). 
\end{enumerate}
\end{prop}
\begin{proof}
As in the proof of~\cite[Lemma III.8.4]{JT}, from Proposition~\ref{prop:YMH-identities}(a)(b) we deduce that
\[
\Delta(\pm f - h) + \frac{|u|^{2}}{\ep^{2}}(\pm f - h)\leq 0.
\]
The first conclusion of part (a) now follows immediately from multiplying both sides with $(\pm f - h)_{+}$, integrating by parts, and recalling that $u \not\equiv 0$. The second conclusion then follows from the strong maximum principle.

For part (b), we recall the following two additional identities.
\begin{align}
|d_{A}u|^{2} &= \left| \frac{d_{A}u - \sqrt{-1}\ast d_{A}u}{2} \right|^{2}  + \left| \frac{d_{A}u + \sqrt{-1}\ast d_{A}u}{2} \right|^{2}. \label{eq:splitting-sum}\\
\ast (d_{A}u \times d_{A}u) &= \left| \frac{d_{A}u + \sqrt{-1}\ast d_{A}u}{2} \right|^{2}  - \left| \frac{d_{A}u - \sqrt{-1}\ast d_{A}u}{2} \right|^{2}. \label{eq:splitting-difference}
\end{align}
Both can be easily verified by direct computation. For instance, to get~\eqref{eq:splitting-difference}, we expand the right-hand side to get
\begin{align*}
\left| \frac{d_{A}u + \sqrt{-1}\ast d_{A}u}{2} \right|^{2}  - \left| \frac{d_{A}u - \sqrt{-1}\ast d_{A}u}{2} \right|^{2} &= \re \langle  d_A u, \sqrt{-1}\ast d_A u\rangle
\end{align*}
Fixing $p \in \Sigma$ and letting $e_1, e_2$ be an orthonormal basis for $T_{p}\Sigma$, we have
\[
(\ast d_A u)_{e_1} = -(d_A u)_{e_2},\ (\ast d_A u)_{e_2} = (d_A u)_{e_1}.
\]
Hence we find that 
\begin{align*}
\re\langle d_A u, \sqrt{-1}\ast d_A u \rangle &= -\re\langle (d_A u)_{e_1}, \sqrt{-1}(d_A u)_{e_2} \rangle + \re\langle (d_A u)_{e_2}, \sqrt{-1}(d_A u)_{e_1} \rangle\\
&= 2\re\langle (d_A u)_{e_1}, \sqrt{-1}(d_A u)_{e_2} \rangle\\
&= (d_A u \times d_A u)_{e_1, e_2} = \ast (d_A u \times d_A u),
\end{align*}
where in getting the second line we used~\eqref{eq:cross-product}. This proves~\eqref{eq:splitting-difference}.

Continuing with part (b), that $f = h$ implies $d_A u = \sqrt{-1}\ast d_A u$ now follows, as in p.97 of~\cite{JT}, from Proposition~\ref{prop:YMH-identities}(a)(b) along with~\eqref{eq:splitting-sum} and~\eqref{eq:splitting-difference}. For the converse, assume $d_A u = \sqrt{-1}\ast d_A u$. Then by~\eqref{eq:splitting-sum} and~\eqref{eq:splitting-difference} we have $|d_A u|^2 = \ast (d_A u \times d_A u)$, which by Proposition~\ref{prop:YMH-identities} implies that
\[
\big( \Delta + \frac{|u|^2}{\ep^2}\big) (f - h) = 0.
\]
Recall by part (a) that $f - h \leq 0$, and hence we get $\Delta (f - h) \geq 0$. Since $\Sigma$ is closed, this means $f-  h$ is constant, and hence
\[
|u|^2(f - h)  = 0,
\]
which forces $f - h$ to vanish identically since, by assumption, $|u|$ is not identically zero.
\end{proof}

For the reader's convenience, we close this section by briefly recalling how to derive~\eqref{eq:Bogol} for closed $\Sigma$. For simplicity we assume $\ep = 1$, as the computation is the same for other cases. The first equality of~\eqref{eq:Bogol} is straightforward. As for the second equality, it suffices to show that 
\[
\int_{\Sigma}\frac{1}{4}\big| d_A u + \sqrt{-1}\ast d_A u \big| + \frac{1}{2}|f + h|^2 d\mu_g - \int_{\Sigma}\frac{1}{4}\big| d_A u - \sqrt{-1}\ast d_A u \big| + \frac{1}{2}|f - h|^2 d\mu_g = \int_{\Sigma}\sqrt{-1}F_A .
\]
To that end, note that
\begin{equation}\label{eq:Bogol-expand}
\frac{1}{2}\int_{\Sigma} | f + h |^{2} -\frac{1}{2}\int_{\Sigma} | f - h |^{2}  d\mu_{g} = \int_{\Sigma} (1 - |u|^2)\sqrt{-1}F_A
\end{equation}
Combining this with~\eqref{eq:splitting-difference} gives
\begin{align}
&\int_{\Sigma}\frac{1}{4}\big| d_A u + \sqrt{-1}\ast d_A u \big| + \frac{1}{2}|f + h|^2 d\mu_g - \int_{\Sigma}\frac{1}{4}\big| d_A u - \sqrt{-1}\ast d_A u \big| + \frac{1}{2}|f - h|^2 d\mu_g\nonumber \\
= & \int_{\Sigma} d_A u \times d_A u - \sqrt{-1}|u|^2 F_A + \int_{\Sigma}\sqrt{-1}F_A \nonumber\\
= &\int_{\Sigma} d(u \times d_A u) + \int_{\Sigma}\sqrt{-1}F_A = \int_{\Sigma}\sqrt{-1}F_A. \label{eq:Bogol-difference}
\end{align}
In the second-to-last equality we used the identity
\[
d(u \times d_A u) = d_A u \times d_A u - \sqrt{-1}|u|^2 F_A,
\]
and the last equality follows because $\Sigma$ is closed by assumption.

\section{Smooth stable solutions of the abelian Yang--Mills--Higgs equations}\label{sec:stable-smooth}
Throughout this section we assume that $\Sigma$ is either the round $S^2$ or a flat $T^{2}$, and that $(u, \nabla_A)$ is a smooth solution to~\eqref{eq:YMH} on all of $\Sigma$. Moreover, we write $h, f, \sigma$ and $\varphi$, respectively, for $h_{(u, \nabla_A)}, f_{(u, \nabla_A)}, \sigma_{(u, \nabla_A)}$ and $\varphi_{(u, \nabla)}$, whose definitions we recall below.
\begin{align*}
h_{(u, \nabla_A)} &= \frac{1 - |u|^{2}}{2\ep^{2}},\\
f_{(u, \nabla_A)} &= \ast \sqrt{-1}F_{A},\\
\sigma_{(u, \nabla_A)} &= d_A u - \sqrt{-1}\ast d_A u,\\
\varphi_{(u, \nabla_A)} &= \sqrt{-1}F_A - \ast h = \ast (f - h).
\end{align*}
Next, we define the following real quadratic form defined over the space of smooth vector fields $X$ on $\Sigma$.
\[
Q(X) := \delta^{2}E_{\ep}(u, \nabla_A)(\sigma_{X}, \iota_{X}\varphi).
\]
As in \cite[Section 10]{BL}, the key step to proving Theorem~\ref{thm:main} consists in computing the trace of $Q$ restricted to the space $\cK$ of Killing vector fields. We begin with the following lemma.
\begin{lemm}\label{lemm:Killing-contraction} 
Let $X \in \cK$. Then the following hold.
\begin{enumerate}
\item[(a)] $J^{1}_{(u, \nabla_A)}(\sigma_{X}, \iota_{X}\varphi) = -2f(d_{A}u)_{X} + 2h\sqrt{-1}(\ast d_{A}u)_{X}$.
\vskip 1mm
\item[(b)] $J^{2}_{(u, \nabla_A)}(\sigma_{X}, \iota_{X}\varphi) = -|d_{A}u|^{2}\iota_{X}(\ast 1) - 2\re\langle (\ast d_{A}u)_{X}, d_{A}u \rangle$.
\end{enumerate}
\end{lemm}
\begin{proof}
We start with (a). First, by Proposition~\ref{prop:YMH-identities}(c) and the fact that $\sqrt{-1}\ast$ commutes with $\Delta_A$, we have
\[
\Delta_A \sigma = -\frac{1}{\ep^{2}}\langle \sigma , u \rangle u + h\sigma + f \sqrt{-1}\ast \sigma, 
\]
and consequently by Lemma~\ref{lemm:contraction}(a) we see that 
\begin{align}
\Delta_{A}(\sigma_{X}) &= -\frac{1}{\ep^{2}}\langle \sigma_X, u \rangle u + h\sigma_X +f (\sqrt{-1}\ast \sigma)_{X} -(F_{A})_{e_i, X}\sigma_{e_i} - \langle dX^{\flat}, d_A\sigma \rangle \nonumber \\
&= -\frac{1}{\ep^{2}}\langle \sigma_X, u \rangle u + h\sigma_X +2f (\sqrt{-1} \ast\sigma)_{X} - \langle dX^{\flat}, d_A\sigma \rangle \label{eq:Deltasigma},
\end{align}
where in getting the second line we used the fact that 
\[
f \sqrt{-1}(\ast \sigma) = -(F_{A})_{e_1, e_2} (\ast \sigma) =  - (F_A)_{e_i,\ \cdot\ }\sigma_{e_i}.
\]
Adding $\frac{|u|^{2} - 1}{2\ep^{2}}\sigma_{X} + \frac{\re\langle u, \sigma_{X} \rangle}{\ep^{2}}u$ to both sides of~\eqref{eq:Deltasigma} gives
\begin{align*}
&\Delta_{A}(\sigma_{X}) - h\sigma_{X} + \frac{\re\langle u, \sigma_{X} \rangle}{\ep^{2}}u\\
=& -\frac{1}{\ep^{2}}\langle \sigma_X, u \rangle u +\frac{1}{2\ep^{2}}\big( \langle u, \sigma_{X} \rangle + \langle \sigma_{X}, u \rangle \big)u + 2f (\sqrt{-1}\ast \sigma)_{X}  - \langle dX^{\flat}, d_A\sigma \rangle\\
=& \frac{1}{2\ep^{2}}\big( \langle u, \sigma_{X} \rangle - \langle \sigma_{X}, u \rangle \big)u + 2f (\sqrt{-1}\ast\sigma)_{X} - \langle dX^{\flat}, d_A\sigma \rangle.
\end{align*}
To see what $J^{1}_{(u, \nabla_A)}(\sigma_{X}, \iota_{X}\varphi)$ actually is, we still need to compute $2\langle \iota_{X}\varphi, \sqrt{-1}d_{A}u \rangle - (d^{\ast}\iota_{X}\varphi)\sqrt{-1}u$. To that end, we recall Proposition~\ref{prop:YMH-identities}(d), which together with Lemma~\ref{lemm:contraction} gives
\begin{align}
d^{\ast}\varphi & = \frac{1}{\ep^{2}}\re\langle \sqrt{-1}u, \sigma \rangle = \frac{\sqrt{-1}}{2\ep^2}\big(\langle u, \sigma \rangle - \langle \sigma, u \rangle \big).\\
\Longrightarrow& -(d^{\ast}\iota_{X}\varphi)\sqrt{-1}u= \frac{1}{2\ep^{2}}\big( \langle \sigma_{X}, u \rangle - \langle u, \sigma_{X} \rangle \big) u - \langle dX^{\flat}, \varphi \rangle \sqrt{-1}u.
\end{align}
Putting everything together, we arrive at 
\begin{align}\label{eq:1st-Jac-almost}
J^{1}_{(u, \nabla_A)}(\sigma_{X}, \iota_{X}\varphi) = 2f(\sqrt{-1}\ast \sigma)_{X} - \langle dX^{\flat}, d_A\sigma + \varphi\sqrt{-1}u \rangle + 2\langle \iota_{X}\varphi, \sqrt{-1}d_{A}u \rangle.
\end{align}
To finish the proof of (a), we note that the second term on the right-hand side vanishes because of Proposition~\ref{prop:YMH-identities}(e). Furthermore, picking an orthonormal frame $\{e_{i}\}$ on $\Sigma$, we compute
\begin{align*}
\langle \iota_{X}\varphi, \sqrt{-1}d_{A}u \rangle &= \varphi_{X, e_{i}}\sqrt{-1}(d_{A}u)_{e_i} = (f - h)(\ast 1)_{X, e_{i}}\sqrt{-1}(d_{A}u)_{e_i}\\
&= -\sqrt{-1}(f - h)(\ast d_{A}u)_{X}.
\end{align*}
Substituting this back into~\eqref{eq:1st-Jac-almost}, recalling the definition of $\sigma$ and observing a cancellation, we obtain
\[
J^{1}_{(u, \nabla_A)}(\sigma_X, \iota_X\varphi) = 2\sqrt{-1}h (\ast d_{A}u)_{X} - 2f (d_{A}u)_{X},
\]
as asserted. 

To prove (b), we first note by Proposition~\ref{prop:YMH-identities}(a)(b) we have
\begin{equation*}
\Delta\varphi = \ast \Delta (f - h) = \frac{1}{\ep^{2}}d_{A}u \times d_{A}u - \frac{|d_{A}u|^{2}}{\ep^{2}}(\ast 1) - \frac{|u|^{2}}{\ep^{2}}\varphi.
\end{equation*}
Thus by Lemma~\ref{lemm:contraction}(d), we have
\begin{align}\label{eq:2nd-Jac-almost}
&\ep^{2}(d^{\ast} d\iota_{X}\varphi)_{e_{j}} + |u|^{2}(\iota_{X}\varphi)_{e_j}\nonumber \\
=& (d_{A}u\times d_{A}u)_{X, e_{j}} - |d_{A}u|^{2}(\ast 1)_{X, e_{j}} + \ep^{2}(d\iota_{X}d^{\ast}\varphi)_{e_j}.
\end{align}
By Proposition~\ref{prop:YMH-identities}(d) and the Leibniz rule, we compute
\[
\ep^{2}(d\iota_{X}d^{\ast}\varphi)_{e_j}= \re\langle \sqrt{-1}(d_{A}u)_{e_j}, \sigma_X \rangle + \re\langle \sqrt{-1}u, (d_A(\sigma_X))_{e_j} \rangle.
\]
Plugging this back into~\eqref{eq:2nd-Jac-almost} and combining $\re\langle \sqrt{-1}(d_{A}u)_{e_j}, \sigma_X \rangle$ with $(d_{A}u\times d_{A}u)_{X, e_{j}}$, we get
\begin{align*}
&\ep^{2}(d^{\ast} d\iota_{X}\varphi)_{e_{j}} + |u|^{2}(\iota_{X}\varphi)_{e_j}\nonumber \\
=& \re\langle \sqrt{-1}(d_{A}u + \sqrt{-1}\ast d_{A}u)_{X}, (d_{A}u)_{e_j} \rangle - |d_{A}u|^{2}(\ast 1)_{X, e_j}+ \re\langle \sqrt{-1}u, (d_A (\sigma_{X}))_{e_j} \rangle.
\end{align*}
Next, noting that
\[
(u\times d_A(\sigma_{X}) + \sigma_{X} \times d_{A}u)_{e_j} = \re\langle \sqrt{-1}u, (d_A(\sigma_{X}))_{e_j} \rangle + \re\langle \sqrt{-1}\sigma_{X}, (d_{A}u)_{e_j} \rangle,
\]
and recalling~\eqref{eq:2nd-Jacobi}, we arrive at the formula asserted in (b). Namely,
\begin{align*}
&\Big(J^{2}_{(u, \nabla_A)}(\sigma_X, \iota_X\varphi)\Big)_{e_j} = -2\re\langle (\ast d_{A}u)_{X}, (d_{A}u)_{e_j} \rangle - |d_{A}u|^{2}(\ast 1)_{X, e_j}.
\end{align*}
The prof of Lemma~\ref{lemm:Killing-contraction} is now complete.
\end{proof}
Now we'd like to use Lemma~\ref{lemm:Killing-contraction} to compute the trace over $\cK$ of the quadratic form $Q$ defined at the beginning of the section, and hence we need to fix an appropriate inner product on $\cK$. To that end we use the fact that if $\Sigma = S^2$ or $T^2$, an inner product on $\cK$ can be chosen so that at each $x \in \Sigma$ there exists an orthonormal basis $X_1, \cdots, X_q$ of $\cK$ such that $X_1(x), X_2(x)$ form an orthonormal basis for $T_{x}\Sigma$, and $X_i(x) = 0$ for $i > 2$. (Here $q = 3$ for $S^2$ and $q = 2$ for $T^2$). The choice of inner products is as follows: if $\Sigma = T^2$, then $\cK$ consists of the parallel vector fields, and we take 
\[
(V, W)_{\cK} := \langle V(x), W(x)\rangle \text{ for }V, W \in \cK,
\]
the choice of $x \in T^2$ being irrelevant because $V, W$ are parallel. On the other hand, if $\Sigma = S^2$, then the $\cK$ is isomorphic to $\mathfrak{so}(3)$ and hence gets an induced inner product from the pairing $(A, B) = \frac{1}{2}\tr(AB^{T})$ on the latter. 

\begin{prop}\label{prop:trace-pointwise} Let $X_{1}, \cdots, X_{q}$ be any orthonormal basis for $\cK$. Then we have
\begin{equation}\label{eq:trace-pointwise}
\sum_{i = 1}^{q}\left[\re\langle J^{1}_{(u, \nabla_A)}(\sigma_{X_i}, \iota_{X_i}\varphi), \sigma_{X_i} \rangle + \langle J^{2}_{(u, \nabla_A)}(\sigma_{X_i}, \iota_{X_i}\varphi), \iota_{X_i}\varphi \rangle\right] = - (f + h)|\sigma|^{2},
\end{equation}
as functions on $\Sigma$.
\end{prop}
\begin{proof}
Fix an arbitrary $x \in \Sigma$ and let $X_{1}, \cdots X_q$ be an orthonormal basis for $\cK$ such that $e_1:= X_{1}(x), e_2: =X_{2}(x)$ is an orthonormal basis for $T_{x}\Sigma$, and $X_{i}(x) = 0$ for $i > 2$. Observe that, to prove the Proposition, it suffices to verify~\eqref{eq:trace-pointwise} for this particular choice of the $X_i$'s. This is because the left-hand side of~\eqref{eq:trace-pointwise} as a function on $\Sigma$ is invariant when we change to another orthonormal basis for $\cK$, since $J^{1}_{u, \nabla_A}$ and $J^{2}_{u, \nabla_A}$ are linear operators.

Now, by our choice of the $X_{i}$'s, and the fact that they are Killing fields, we can apply Lemma~\ref{lemm:Killing-contraction} to see that the left-hand side of~\eqref{eq:trace-pointwise} evaluated at $x$ is equal to
\begin{align*}
&\left[ 2h \re\langle \sqrt{-1}\ast d_{A}u, \sigma \rangle - 2f\re\langle d_{A}u, \sigma \rangle \right] -(f - h)(\ast 1)_{e_i, e_j}\left[ |d_{A}u|^{2}(\ast 1)_{e_i, e_j} + 2\re\langle (\ast d_{A}u)_{e_{i}}, (d_{A}u)_{e_j} \rangle\right]\\
&= \left[ 2h \re\langle \sqrt{-1}\ast d_{A}u, \sigma \rangle - 2f\re\langle d_{A}u, \sigma \rangle \right],
\end{align*}
where in getting the last line we made a cancellation with the help of the following identities:
\[
\sum_{i, j = 1}^{2}(\ast 1)_{e_i, e_j}\re\langle (\ast d_{A}u)_{e_i}, (d_{A}u)_{e_{j}} \rangle = -|d_{A}u|^{2}\ ;\ \sum_{i, j = 1}^2 (\ast 1)_{e_i, e_j}(\ast 1)_{e_i, e_j} = 2.
\]
The first identity can be checked by recalling that, from the definition of the Hodge star operator, we have
\begin{align*}
(\ast d_A u)_{e_1} &= -(d_A u)_{e_2},\\
(\ast d_A u)_{e_2} &= (d_A u)_{e_1}.
\end{align*}
Recalling the definition of $\sigma$ and expanding the inner products, we get
\[
2h \re\langle \sqrt{-1}\ast d_{A}u, \sigma \rangle - 2f\re\langle d_{A}u, \sigma \rangle = (f + h) \big( 2\re\langle \sqrt{-1}\ast d_{A}u, d_{A}u \rangle  - 2|d_{A}u|^{2}\big) = -(f + h)|\sigma|^{2},
\]
as asserted.
\end{proof}
\begin{coro}\label{coro:trace} We have
\[
\tr_{\cK}Q = -\int_{\Sigma}(f + h)|\sigma|^{2}d\mu_{g}.
\]
\end{coro}
\begin{proof}
By definition of $Q$, we have
\[
\tr_{\cK}Q =\int_{\Sigma}\sum_{i = 1}^{q}\left[\re\langle J^{1}_{(u, \nabla_A)}(\sigma_{X_i}, \iota_{X_i}\varphi), \sigma_{X_i} \rangle + \langle J^{2}_{(u, \nabla_A)}(\sigma_{X_i}, \iota_{X_i}\varphi), \iota_{X_i}\varphi \rangle\right] d\mu_{g},
\]
where $X_{1}, \cdots, X_{q}$ is an orthonormal basis for $\cK$. The result now follows from Proposition~\ref{prop:trace-pointwise}.
\end{proof}
\begin{proof}[Proof of Theorem~\ref{thm:main}]
We first deal with the case where (H) is assumed. By Proposition~\ref{prop:global-regularity} we may assume that $(u, \nabla_A)$ is smooth. Since $(u, \nabla_A)$ is stable by assumption, from Corollary~\ref{coro:trace} we have
\[
\int_{\Sigma}(f + h)|\sigma|^{2}d\mu_g = -\tr_{\cK}Q \leq 0.
\]
By Proposition~\ref{prop:estimates}(a), we have that $f + h \geq 0$, so the above inequality implies that 
\[
(f+ h) |\sigma|^{2} = 0 \text{ everywhere on }\Sigma.
\]
Now recall, again by Proposition~\ref{prop:estimates}(a), that $f + h$ is either identically zero or everywhere positive. In the former case, we are done by Proposition~\ref{prop:estimates}(b). In the latter case, we get $\sigma \equiv 0$ and again we are done by Proposition~\ref{prop:estimates}(b).

To prove Theorem~\ref{thm:main} assuming (H') instead, we need only consider the case $u \equiv 0$, since otherwise we reduce to the previous case. Note that since $u \equiv 0$, we have from~\eqref{eq:YMH} that 
\[
d^{\ast}F_A = 0.
\]
Consequently $\sqrt{-1}F_A$ is a real harmonic $2$-form on $\Sigma$ and hence a constant multiple of $d\mu_g$. Recalling~\eqref{eq:deg-vortex}, we see that we must have
\begin{equation}\label{eq:harmonic-curvature}
\ast\sqrt{-1}F_{A} = \frac{d}{2}.
\end{equation}
In particular, in the case $|d| = \ep^{-2}$, since $u \equiv 0$, we have
\[
\ast\sqrt{-1}F_A = \pm\frac{1}{2\ep^2} =\pm\frac{1 - |u|^2}{2\ep^2}.
\]
Hence $(u, \nabla_A)$ solves~\eqref{eq:vortex} or~\eqref{eq:anti-vortex} by Proposition~\ref{prop:estimates}(b).

On the other hand, if $|d| < \ep^{-2}$, then $u \equiv 0$ contradicts stability and thus cannot occur. Indeed, observe that in this case the second variation formula~\eqref{eq:2nd-variation} gives
\begin{equation}\label{eq:reduced-2nd-variation}
\delta^2 E_{\ep}(0, \nabla_A)(v, 0) = \int_{S^2} 2|d_A v|^2 - \frac{|v|^2}{\ep^2}.
\end{equation}
Now because $\Sigma = S^2$ and because of~\eqref{eq:harmonic-curvature}, we can invoke~\cite[Theorem 5.1]{Ku} to see that the lowest eigenvalue of $d_{A}^{\ast}d_A$ acting on sections of $L$ is equal to $\frac{|d|}{2}$. That is
\begin{equation}\label{eq:eigenvalue}
\inf \left\{ \int_{S^2}|d_A v|^2 d\mu_g \ \big|\ v \in \Omega^0(L),\ \int_{S^2}|v|^2 d\mu_g = 1\right\} = \frac{|d|}{2}.
\end{equation}
This together with~\eqref{eq:reduced-2nd-variation} and the assumption $|d| < \ep^{-2}$ shows that $(0, \nabla_A)$ is unstable, a contradiction.
\end{proof}
\begin{rmk}
Here we clarify Remark~\ref{rmk:main}(2). Note that if~\eqref{eq:harmonic-curvature} holds for $\nabla_A$, then $(0, \nabla_A)$ solves~\eqref{eq:YMH}. Moreover, since $\Sigma = S^2$, we get~\eqref{eq:eigenvalue} thanks to~\cite{Ku}. By~\eqref{eq:2nd-variation}, this implies that for all $(v, a) \in \Omega^{0}(L) \times \Omega^1(S^2)$, we have
\begin{align*}
\delta^2 E_{\ep}(0, \nabla_A)(v, a) &= \int_{S^2} 2\ep^2 |da|^2 + 2|d_A v|^2 - \frac{|v|^2}{\ep^2} \geq \int_{S^2}2|d_A v|^2 - \frac{|v|^2}{\ep^2}\\
&\geq \Big( |d| - \frac{1}{\ep^2}\Big)\int_{S^2} |v|^2 \geq 0,
\end{align*}
provided $|d| \geq \ep^{-2}$. In other words, $(0, \nabla_A)$ is stable when $|d| \geq \ep^{-2}$. However, for $(0, \nabla_A)$ to be a solution to~\eqref{eq:vortex} or~\eqref{eq:anti-vortex}, we must have $|d| = \ep^{-2}$. Thus, there exist stable solutions to~\eqref{eq:YMH} on $S^2$ which do not satisfy either of the vortex equations when $|d| > \ep^{-2}$, because we can certainly find a connection on $L$ whose curvature satisfies~\eqref{eq:harmonic-curvature}, for example by looking at the Hodge decomposition of $F_0$, the curvature of the background connection $\nabla_0$. 
\end{rmk}
\bibliographystyle{amsalpha}
\bibliography{main}
\end{document}